\documentclass[reqno,12pt]{article}
\usepackage{amssymb,amsthm,amsmath,amsfonts}
\usepackage{epsfig}

\theoremstyle{plain}
\begingroup

\newtheorem{thm}{Theorem}[section]
\newtheorem{theorem}{Theorem}[section]

\newtheorem{corollary}[thm]{Corollary}
\newtheorem{lemma}[thm]{Lemma}

\endgroup
\theoremstyle{definition}
\newtheorem{defn}{Definition}[section]

\newtheorem{remark}[defn]{Remark}

\newtheorem{example}[defn]{Example}

\theoremstyle{remark}

\numberwithin{equation}{section}
\numberwithin{figure}{section}

\DeclareMathOperator*{\res}{\mathrm{Res}}

\def\I{\mathrm{i}}

\def\D{{\mathbb D}}

\def\C{{\mathbb C}}
\def\P{{\mathbb P}}
\def\Z{{\mathbb Z}}
\def\N{{\mathbb N}}

\def\ol{\overline}

\begin{document}

\title{A field theoretic operator model and Cowen-Douglas class}

\author{
Bj\"orn Gustafsson\textsuperscript{1},
Mihai Putinar\textsuperscript{2}\\
}

\maketitle

\begin{abstract}
In resonance to a recent geometric framework proposed by Douglas and Yang, a functional model for certain linear bounded operators 
with rank-one self-commutator acting on a Hilbert space is developed. By taking advantage of the refined existing theory of the 
principal function of a hyponormal operator we transfer the whole action outside the spectrum, on the resolvent of the 
underlying operator, localized at a distinguished vector. The whole construction turns out to rely on an elementary algebra body 
involving
analytic multipliers and Cauchy transforms. A natural field theory interpretation of the resulting resolvent functional model is proposed.

\end{abstract}

\bigskip
\begin{center}
Ron Douglas, in memoriam
\end{center}
\bigskip

\noindent {\it Keywords:} Hyponormal operator, exponential transform, Cauchy transform, ideal fluid flow.

\noindent {\it MSC Classification:} Primary: 47B20; Secondary: 30A31,76C05.

 \footnotetext[1]
{Department of Mathematics, KTH, 100 44, Stockholm, Sweden.\\
Email: \tt{gbjorn@kth.se}}
\footnotetext[2]
{Department of Mathematics, UCSB, Santa Barbara, CA 93106-3080, USA and School of Mathematics, Statistics and Physics,
Newcastle University, Newcastle upon Tyne, NE1 7RU, UK.\\
Email: \tt{mputinar@math.ucsb.edu, mihai.putinar@ncl.ac.uk}}


\section{Introduction}\label{sec:intro}

Functional models open deep insights into classification, spectral analysis and refined structure of special classes of linear operators.
Among the most natural examples are: the spectral theorem for self-adjoint or unitary transformations, Rota's model, Livsic' model for dissipative operators, Sz.-Nagy and Foia\c{s} model for contractive operators,
singular integral models for close to normal operators, the sheaf model of decomposable operators. The starting point in all cases is the ``diagonalization'' of a linear and bounded operator 
${\rm T}$ acting on a Hilbert space $H$, via its resolvent:
$$ 
{\rm T} ({\rm T}-z)^{-1}x = z ({\rm T}-z)^{-1}x + x, \ \ x \in H.
$$
Thus the action of ${\rm T}$ becomes multiplication by the complex variable (a much more standard operation), in a properly chosen function space, modulo some canonical quotient. 
One cannot underestimate this very simple idea, consult for instance the survey \cite{Fuhrmann-1990}.

One step further, classifying operators via the local behavior of their (left, or right) resolvent, is an equally natural step to take. In a groundbreaking work Cowen and Douglas \cite{Cowen-Douglas-1978} proposed the following classification scheme.
Assume that the spectrum of ${\rm T}$ contains an open set $\Omega$, and that for every $\lambda \in \Omega$, the operator ${\rm T}-\lambda= {\rm T}-\lambda {\rm I}$ 
is surjective and its kernel has constant dimension $d$. Moreover,
assume that the linear span of $\ker ({\rm T}-\lambda)$, $\lambda \in \Omega,$ spans the entire Hilbert space $H$. Then curvature type invariants of the complex hermitian vector bundle 
$\ker ({\rm T}-\lambda)$, $\lambda \in \Omega,$ form a complete set of unitary invariants for ${\rm T}$. This ingenious bridge between operator theory and complex hermitian geometry  
produced profound results and continues to attract a lot of attention. See \cite{Douglas-Yang-2016} for some recent developments.

The typical example in Cowen and Douglas (in short CD) class is the adjoint ${\rm S}^\ast$ of the unilateral shift operator ${\rm S} \in \mathcal{L}(\ell^2(\N))$. On an orthonormal basis 
$(e_n)_{n=0}^\infty$ the shift ${\rm S}$ acts by translation
${\rm S}e_n = e_{n+1}, \ n \geq 0.$ One easily verifies that ${\rm S}^\ast$ is in Cowen and Douglas class with $\Omega$ equal the unit disk and multiplicity $d$ equal to $1$. As a matter of fact, adjoints of multiplication by the complex variable $z$ on 
``natural'' Hilbert spaces of analytic functions in a domain $\Omega$ possess the same CD property.

But one can turn the page and go beyond spaces of analytic functions. A notable class is offered by singular integral operators with a Cauchy kernel type.  More precisely, consider on Lebesgue space
$L^2([0,1], dt)$ a linear operator of the form
$$ 
({\rm T}f)(x) = xf(x) + a(x) \int_0^1 \frac{ a(t) f(t)}{t-x} dt
$$
with a bounded, measurable and real valued multiplier $a$. Since Hilbert transform is bounded on Lebesgue space, we identify the real and imaginary parts of ${\rm T}={\rm X}+\I{\rm Y}$ as
$$ 
({\rm X}f)(x) = xf(x),  \ \ \ ({\rm Y}f)(x) = -\I a(x) \int_0^1 \frac{ a(t) f(t)}{t-x} dt.
$$ 
The notable feature in this example is the rank-one non-negative commutator:
$$ 
[{\rm T},{\rm T}^\ast] = [{\rm X}+\I{\rm Y}, {\rm X}-\I{\rm Y}] = -2\I [{\rm X},{\rm Y}] = 2 a \langle \cdot, a\rangle.
$$
While this is not Heisenberg canonical commutation relation in quantum mechanics, it is close, in the sense that one can realize with linear and bounded operators the commutator inequality
$[{\rm T},{\rm T}^\ast] \geq 0$. And the comparison does not stop here. Such operators are called co-hyponormal and they display some remarkable spectral behavior. 

If the function $a$ above is in addition continuous, one proves with specific techniques that ${\rm T}$ belongs to Cowen Douglas class, with multiplicity $d=1$. This observation goes back several decades ago to Clancey and Wadhwa \cite{Clancey-Wadhwa-1983}.

The aim of the present article is to adopt from hyponormal operator theory some basic constructions, solely relying on Cauchy integrals of measures, for constructing a field theoretic model to a class of operators sharing the Cowen Douglas property. We emphasise the nature and novelty of the underlying Hilbert space structure where the operators naturally act. The note is a direct continuation of our recent Lecture Notes \cite{Gustafsson-Putinar-2017}.

An informal overview of our computations is the following: we populate the non-essential part of the spectrum with vortices, sources and sinks and we transfer the action of the underlying operator outside the spectrum, on the velocity fields generated by these singularities. Although in general fields do not faithfully locate their sources, our operation is not loosing information. This trend resonates to the ``remote sensing'' concept.

  
\section{Preliminaries}\label{sec:preliminaries} 

A linear bounded operator ${\rm S} \in \mathcal{L}(H)$ acting on Hilbert space $H$ is called {\it hyponormal} if the self-commutator ${\rm D} = [{\rm S}^\ast,{\rm S}]$ is non-negative, as an operator. The vectors in the range of
$\sqrt{{\rm D}}$ are special, due to the fact, verified by elementary means, that the resolvent 
$$ 
\rho_h(z) = ({\rm S}^\ast - \bar{z})^{-1} \sqrt{{\rm D}} h
$$ 
admits a bounded measurable extension across the spectrum of ${\rm S}$. Indeed, fix $z \in \C$ and a vector $h$. Since
$$ \| ({\rm S}-z)h \|^2 - \|({\rm S}^\ast - \bar{z})h\|^2 = \| \sqrt{\rm D}h \|^2,$$
we infer there exists a contractive operator ${\rm T}(z)$ satisfying
$$ \sqrt{\rm D} = {\rm T}(z) ({\rm S}-z).$$
To resolve any ambiguity, we may impose the condition 
$$ {\rm T}(z) {\rm ker}({\rm S}^\ast - \bar{z}) = 0.$$
By taking adjoints we find
$$ ({\rm S}^\ast - \bar{z}) {\rm T}(z)^\ast = \sqrt{\rm D},$$
and we call ${\rm T}(z)^\ast h = ({\rm S}^\ast - \bar{z})^{-1} \sqrt{{\rm D}}h$
the localized resolvent of ${\rm S}$, at the vector $ \sqrt{{\rm D}}h$.

Consequently, the continuous Gram matrix type kernel
$\langle \rho_h(z), \rho_k(w) \rangle$, $z,w \in \C$, $h,k \in {\rm Ran} \sqrt{{\rm D}}$, determines the irreducible part of ${\rm S}$, that is the compression of ${\rm S}$ to the reducing space which contains the range of $\sqrt{{\rm D}}$. This is the well paved road of 
``diagonalizing''  ${\rm S}^\ast$ and obtaining a functional model. Henceforth we assume that the operator ${\rm S}$ is irreducible on $H$, which is equivalent to the fact that there is no normal direct summand in ${\rm S}$.

Indeed, for any analytic functions $f,g$ defined in a neighborhood of the spectrum $\sigma({\rm S})$, one finds by Riesz-Dunford calculus:
$$
\langle g({\rm S})^\ast \sqrt{{\rm D}}h, f({\rm S})^\ast \sqrt{{\rm D}} k \rangle = 
$$ 
$$
\frac{1}{4 \pi^2} \int_{\gamma \times \gamma} \langle \ol{g(z)}  ({\rm S}^\ast - \bar{z})^{-1} \sqrt{{\rm D}} h, \ol{f(w)}  ({\rm S}^\ast - \bar{w})^{-1} \sqrt{{\rm D}} k \rangle dw d\bar{z} = 
$$ 
$$
 \frac{1}{4 \pi^2} \int_{\gamma \times \gamma} \langle \rho_h(z), \rho_k(w) \rangle f(w) \ol{g(z)} dw d\bar{z},
$$
where $\gamma$ is a system of Jordan curves surrounding the spectrum of ${\rm S}$. Already this defines a non-negative inner product on the space of vector valued functions spanned by 
$f(z) \sqrt{{\rm D}}h$, and by reversing the above computations, one finds a correlation, continuous kernel $\langle \rho_h(z), \rho_k(w) \rangle$ which diagonalizes the operator ${\rm S}^\ast$. 

The above scheme is well known and widely used in operator theory, starting with Rota model \cite{Rota-1960} and not disjoint of Cowen and Douglas configuration. What is special in our setting is 
the measurability of the extensions of the resolvents across the spectrum. Hence we can play the distribution game and notice that the eigenvector equation
$$ 
({\rm S}^\ast - \bar{z}) \frac{\partial}{\partial z} \rho_h(z) = 0,
$$
makes perfect sense, in the weak sense. And moreover, the vector valued eigen-distributions $\frac{\partial}{\partial z} \rho_h(z)$ span the whole space, that is
the vectors $\frac{\partial}{\partial z} \rho_h(z)(\phi)$ where $\phi = \phi(z,\bar{z})$ runs among all test functions in $\C$, are dense in $H$. In principle, if we know enough many eigenvectors of an operator 
${\rm S}$ one can classify it up to unitary equivalence class, one can build a natural functional model which substitutes ${\rm S}$ by the multiplication by the variable, and so on. 

The simplest case of irreducible hyponormal operators with rank-one self-commutator
$$ 
[{\rm S}^\ast,{\rm S}] = \xi \langle \cdot, \xi \rangle
$$
allows some simplifications, it is way better understood and it makes the subject of our note. First of all the diagonalizing kernel 
$\langle \rho_h(z), \rho_k(w) \rangle$ is reduced in this case to a scalar correlation matrix type, still continuous:
$$ 
K_S(w,z) = K(w,z) = \langle  ({\rm S}^\ast - \bar{z})^{-1} \xi,  ({\rm S}^\ast - \bar{w})^{-1} \xi \rangle
$$ 
makes perfect sense for all values of $z, w \in \C$ and one proves that in fact that $K(w,z)$ is separately continuous, besides being hermitian. Moreover, as noted above, the sole eigen-distribution 
$\frac{\partial}{\partial z} ({\rm S}^\ast - \bar{z})^{-1} \xi$ spans the whole underlying Hilbert space. Note that this $H$-valued distribution is supported by 
the spectrum of ${\rm S}$. One proves that outside the essential spectrum of ${\rm S}$ the distribution $\frac{\partial}{\partial z} ({\rm S}^\ast - \bar{z})^{-1} \xi$ is an analytic frame in the hermitian line bundle
$\ker({\rm S}^\ast-\bar{z})$, thus meeting the heart of Cowen and Douglas' framework.

Once the spanning eigen-distribution is known, the natural functional model is straightforward: define on Schwartz space ${\mathcal S}(\C)$ the inner product
$$ 
(\psi, \phi) = \int_{\C \times \C} K(w,z)\frac{ \partial \ol{\phi}}{\partial z} \frac{ \partial \psi}{\partial \bar{w}} \,\,d A \otimes  d A,
$$
where $dA$ stands for Lebesgue area measure. That is
$$ 
(\phi,\psi) = \langle [\frac{\partial}{\partial z} ({\rm S}^\ast - \bar{z})^{-1} \xi] (\phi), [\frac{\partial}{\partial z} ({\rm S}^\ast - \bar{z})^{-1} \xi] (\psi) \rangle = \langle {\rm R}(\phi), {\rm R} (\psi) \rangle,
$$
where 
$$ 
{\rm R} : {\mathcal S}(\C) \longrightarrow H, \quad{\rm R} (\phi) = [\frac{\partial}{\partial z}({\rm S}^\ast - \bar{z})^{-1} \xi](\phi)
$$
is a continuous anti-linear map. The intertwining property
$$ 
{\rm R} (\bar{z} \phi) = {\rm S}^\ast {\rm R}(\phi)
$$
induces the unitary equivalence between the Hilbert space completion of ${\mathcal S}(\C)$ with respect to the norm induced by the inner product $(\cdot, \cdot)$ and $H$. In this unitary transformation the operator ${\rm S}^\ast$ becomes multiplication by $\bar{z}$.

Two remarks are in order. First, the operator ${\rm S}$ is not normal, hence it is not represented in the above model by the multiplication by $z$. A correction term involving Cauchy transforms appear, as it will be demonstrated 
in the present article. Second, the correlation/diagonalization kernel $K_S$ has a precise functional description involving again a (double) Cauchy transform and depending solely on a
complete unitary invariant of the operator ${\rm S}$, called the {\it principal function}. Without entering into details, which appear with complete proofs in \cite{Martin-Putinar-1989}, there is a bijective correspondence between irreducible 
hyponormal operators ${\rm S}$ with rank-one self-commutator and classes in $L^1(\C,dA)$ of measurable functions $\rho : \C \longrightarrow [0,1]$, established by the formula
$$ 
K_S(w,z) = \langle  ({\rm S}^\ast - \bar{z})^{-1} \xi,  ({\rm S}^\ast - \bar{w})^{-1} \xi \rangle =  
$$ 
$$
1- \exp[-\frac{1}{\pi}\int_\C \frac{\rho (\zeta)dA(\zeta)}{(\zeta-w)(\bar{\zeta}-\bar{z})}], \quad z,w\in\C.
$$
There is much to say, old and new, about the exponential transform appearing in the last expression, see for instance our recent Lecture Notes \cite{Gustafsson-Putinar-2017}.

The scope of this article is to enter into the refined structure of the functional space ${\mathcal S}(\C)$, endowed with the inner product $(\cdot, \cdot)$ and its Hilbertian completion, from a purely function theoretic perspective, touching field theory
interpretations. 

We will soon specialize to principal functions which are characteristic functions of bounded open sets $\Omega$, and sometimes with some boundary regularity. In that case, the corresponding operator ${\rm S}^\ast$
belongs to Cowen-Douglas class, with the set $\Omega$ carrying a hermitian line bundle of eigenvectors. In this respect, we provide in the rest of this note a novel, constructive approach to a non-standard collection of CD operators. On the other hand, the class of linear operators isolated by Cowen and Douglas is much larger, even when restricting only to rank one vector bundles of eigenvalues. Indeed, many cyclic subnormal operators share the Cowen-Douglas spanning property, while the only one which in addition is hyponormal of 
rank-one self-commutator is the unilateral shift, see \cite{Martin-Putinar-1989} for details.


\subsection{Notation}

We list below some non-standard notation appearing in the present article.
\begin{itemize}

\item $\mathcal{H}(\Omega)$:  a standard functional space for (co)hyponormal operators.

\item $\mathcal{O}(E)$: the set of germs of analytic functions on an arbitrary set $E\subset \P$, $\P$ denoting the Riemann sphere. 

\item $\mathcal{O}(\P\setminus \Omega)_0 = \{f\in\mathcal{O}(\P\setminus \Omega): f(\infty) =0)\}$ ($\Omega\subset\C$ a bounded domain).
\item $\overline{\mathcal{O}}(\P\setminus \Omega)_0 = $ completion of $\mathcal{O}(\P\setminus \Omega)_0$
with respect to relevant norm.
\item ${\rm C}[\mu](z)= -\frac{1}{\pi}\int\frac{d\mu(\zeta)}{\zeta-z}$, 
the Cauchy transform of a measure (or distribution) $\mu$ with ${\rm supp\,}\mu\subset\overline{\Omega}$. 

\item ${\rm C}^{\rm ext}[\mu]={\rm C}[\mu]\big|_{\P\setminus {\Omega}}$.

\item $S(z)$ is defined to be the function $S(z)=\bar{z}$ on $\partial\Omega$. If $\partial\Omega$ is analytic then $S(z)$ extends analytically
to a neighborhood of $\partial\Omega$ and is called the Schwarz function for $\partial\Omega$.  

\item If a function $f$ is defined on $\partial\Omega$, then its inner and outer Cauchy transforms $f_\pm$ are defined by (\ref{plusminus}).
In particular $S_\pm (z)$ are defined this way, and $S_-$ then equals the exterior Cauchy transform of the domain: 
$S_-={\rm C}^{\rm ext}[\chi_\Omega\, dA]$.

\end{itemize}



\section{An analytic functional model based on Cauchy transforms}\label{sec:model}

\subsection{Definitions in terms of the exponential transform}

In the main body of our article, to start with this section, we turn our attention to the action outside the spectrum of a hyponormal operator ${\rm S}$ with rank-one self-commutator
on the generating resolvent, primitive of the eigen-distribution $\frac{\partial}{\partial z} ({\rm S}^\ast - \bar{z})^{-1} \xi$. Obviously this operation diagonalizes ${\rm S}^\ast$,
following the conventions of the preceding section, modulo constants of integration. Henceforth we closely follow this scenario, as we shall see with clear benefits for understanding
a versatile functional model. 
To conform with notations used in \cite{Gustafsson-Putinar-2017} we adopt ${\rm Z}$ for the prior ${\rm S}^\ast$, and we shall work exclusively in a functional model where this cohyponormal
operator ${\rm Z}$ becomes (after an additional conjugation of the complex plane) multiplication with the independent variable $z$. The vector $\xi$ becomes the function identically one,
denoted ${\bf 1}$.
The above mentioned eigen-distribution then takes the form $\frac{\partial}{\partial \bar{z}} ({\rm Z} - {z})^{-1} {\bf 1}$,
which will be further transformed into $\frac{\partial}{\partial \bar{z}} ({\rm Z}_\mathcal{O} - {z})^{-1} {S_-}$ in the model developed in the present paper.

We start with the inner product
\begin{equation}\label{innerproduct0}
\langle f,g\rangle =\frac{1}{\pi^2}\int_\C\int_\C (1-E_\rho(z,w))
\frac{{\partial} {f}}{\partial \bar{z}}\, \overline{\frac{\partial{{g}}}{\partial \bar{w}}} \,dA(z) dA(w),
\end{equation}
defined initially for smooth functions $f$, $g$ with compact support in $\C$. Here $0\leq \rho\leq 1$ is the principal function for the underlying cohyponormal operator,
the support of  $\rho$ being the spectrum of the operator. The exponential transform of $\rho$ is
\begin{equation}\label{E}
E_\rho(z,w)=\exp[-\frac{1}{\pi}\int_\C \frac{\rho (\zeta)dA(\zeta)}{(\zeta-z)(\bar{\zeta}-\bar{w})}], \quad z,w\in\C.
\end{equation}

When $\rho$ is the characteristic function of a domain, say $\rho=\chi_\Omega$, the inner product (\ref{innerproduct0}) reduces, upon partial integration, to an
integral over just $\Omega$ and with respect to the kernel (inner exponential transform)
\begin{equation}\label{HE}
H(z,w)=-\frac{\partial^2 E(z,w)}{\partial \bar{z}\partial w}, \quad z,w\in\Omega,
\end{equation}
which is analytic in $z$, anti-analytic in $w$. (The right member in (\ref{HE}) vanishes outside ${\Omega}\times{\Omega}$). 
After completion and passing to a quotient space (there will be many functions with zero norm)
there arises a Hilbert space $\mathcal{H}(\Omega)$. The objects in this space are equivalence classes of distributions (or possibly even more general objects), 
which  may have support reaching out to the boundary of $\Omega$. See \cite{Gustafsson-Putinar-2017} for some details and further references. 
The space $\mathcal{H}(\Omega)$ turns out to be very strange when considered as a function space and it
is  desirable to have another model with better properties in this respect, for example a space consisting of analytic functions and for which
there are continuous point evaluations. 
In this paper we introduce such an analytic model. The model reminds of Rota's model \cite{Rota-1960}, and in the case of hyponormal operators the Pincus-Xia-Xia analytic model \cite{Pincus-Xia-Xia-1984}; however, we deal in the sequel with analytic functions defined on the complement of the spectrum of the underlying operator.

Very briefly, the new space consists of Cauchy transforms in $\P\setminus\Omega$ of objects in $\mathcal{H}(\Omega)$.
Said in another way, one may start from (\ref{innerproduct0}) and consider it as a product between primitives $\tilde{f},\tilde{g}$ (with respect to $\bar{z}$)
of the functions $f$, $g$ involved, i.e. $f=\frac{\partial\tilde{f}}{\partial\bar{z}}$, $g=\frac{\partial\tilde{g}}{\partial\bar{z}}$. Specifically, these are to be taken to be the respective Cauchy transforms:
$\tilde{f}={\rm C}[f]$, $\tilde{g}={\rm C}[g]$. Then, with $\rho=\chi_\Omega$ and writing here the original inner product as $\langle \cdot, \cdot\rangle_{\mathcal{H}(\Omega)}$ for clarity, we have
\begin{equation*}\label{innerproduct3}
\langle \frac{{\partial} \tilde{f}}{\partial \bar{z}}\,, {\frac{\partial{\tilde{g}}}{\partial \bar{w}}}\rangle_{\mathcal{H}(\Omega)}=
\frac{1}{\pi^2}\int_\C\int_\C (1-E(z,w))
\frac{{\partial^2} \tilde{f}}{\partial \bar{z}^2}\, \overline{\frac{\partial^2{\tilde{g}}}{\partial \bar{w}^2}} \,dA(z) dA(w)=
\end{equation*}
$$
=\frac{1}{\pi^2}\int_\Omega\int_\Omega H(z,w)\frac{{\partial} \tilde{f}}{\partial \bar{z}}\, \overline{\frac{\partial{\tilde{g}}}{\partial \bar{w}}} \,dA(z) dA(w)=
$$
\begin{equation}\label{innerproduct1}=\frac{1}{4\pi^2}\int_{\partial\Omega}\int_{\partial\Omega} H(z,w) \tilde{f} (z)\, \overline{\tilde{g}(w)} \,dzd\bar{w}
=\langle \tilde{f},\tilde{g}\rangle_{\mathcal{O}}.
\end{equation}
Here the last equality is the definition of the new inner product $\langle \tilde{f}$, $\tilde{g}\rangle_{\mathcal{O}}$. In the final boundary integral $\partial\Omega$ has to be approximated from inside via
an exhaustion of subdomains of $\Omega$ because the kernel $H(z,w)$ has a singularity on the boundary. See Section~\ref{sec:boundary} for some details on this.

In general we do not assume much of regularity for $\partial\Omega$. Being it locally a Lipschitz graph will mostly be enough, even though, at some (specified) occasions, 
it will be useful to have it analytic. 

The definition of the original inner product, in $\mathcal{H}(\Omega)$,  immediately applies to distributions (even analytic functionals) with compact support (carrier) in $\Omega$, as
$$
\langle \mu,\nu\rangle=\frac{1}{\pi^2} (\mu\otimes\bar{\nu}) H,
$$
and the new inner product accordingly applies to all functions $\mathcal{O}(\P\setminus \Omega)_0$. 
By the above computation, the correspondence preserves the inner product: 
$$
\langle \mu,\nu\rangle=\langle  {\rm C}^{\rm ext}[\mu],  {\rm C}^{\rm ext}[\nu]\rangle_\mathcal{O}.
$$
Thus it gives rise to a unitary isomorphism between the completions of the spaces involved, in one direction given by the Cauchy transform:
$$
{\rm C}^{\rm ext}:\mathcal{H}(\Omega)\to\overline{\mathcal{O}}(\P\setminus \Omega)_0,
$$
$$
\mu \mapsto {\rm C}^{\rm ext}[\mu]={\rm C}[\mu]\big|_{\P\setminus\Omega}.
$$

As is implicit in the above discussion, it does not matter which object $\mu$ in a given equivalence class in the quotient space $\mathcal{H}(\Omega)$ is chosen
when taking the Cauchy transform.
This amounts to saying that since the inner product (\ref{innerproduct1}) only depends on the values of $\tilde{f}$ and $\tilde{g}$ on $\partial\Omega$, it does not matter
in which way these boundary functions are prolonged into $\Omega$.

In the other direction, if $f\in\mathcal{O}(\P\setminus \Omega)_0$ then it is possible to extend $f$ to a smooth function $\tilde{f}$ in $\P$ by modifying $f$ only in a compact subset of $\Omega$.
Then the Cauchy transform of ${\partial \tilde{f}}/{\partial\bar{z}}$ equals $\tilde{f}$, everywhere, in particular it agrees with $f$ outside $\Omega$.
The so obtained derivative ${\partial \tilde{f}}/{\partial\bar{z}}$ can be identified with an element in $\mathcal{H}(\Omega)$, as
$$
[\frac{\partial \tilde{f}}{\partial\bar{z}}]\in\mathcal{H}(\Omega),
$$
the bracket here denoting equivalent class. Passing to the completion  $\overline{\mathcal{O}}(\P\setminus {\Omega})_0$
this gives the isomorphism in the other direction:
$$
[{\frac{\partial}{\partial\bar{z}}}]: \overline{\mathcal{O}}(\P\setminus \Omega)_0\to \mathcal{H}(\Omega),
$$
$$
 f \mapsto [\frac{\partial \tilde{f}}{\partial\bar{z}}].
$$

It is easy to see that point evaluations are continuous functionals on $\mathcal{O}(\P\setminus {\Omega})_0$ with respect to the new norm. In fact, 
the corresponding reproducing kernel will be explicitly given in Section~\ref{sec:kernel}. Therefore
$$
\mathcal{O}(\P\setminus {\Omega})_0\subset \overline{\mathcal{O}}(\P\setminus \Omega)_0\subset {\mathcal{O}}(\P\setminus \overline{\Omega})_0,
$$
and $\overline{\mathcal{O}}(\P\setminus \Omega)_0$ becomes a true function space of analytic functions.


\subsection{Interior and exterior Cauchy transforms}\label{sec:intext}

When a measure $\mu$ has support on $\partial\Omega$ and is of the form 
$$
d\mu(z)=\frac{1}{2\I}f (z)dz,
$$ where $f$ is a continuous (say) function on $\partial\Omega$, then its Cauchy transform defines one analytic function in the exterior,
$f_-\in  {\mathcal{O}}(\P\setminus \overline{\Omega})_0$, and one in the interior, $f_+\in  {\mathcal{O}}({\Omega})$. More precisely we define $f_\pm$ so that
\begin{equation}\label{plusminus}
{\rm C}[\mu] (z)=-\frac{1}{2\pi \I}\int_{\partial\Omega}\frac{f(\zeta)d\zeta}{\zeta-z}=
\begin{cases}
\quad -f_+(z), \quad z\in {\Omega},\\ \quad\,\,\, f_-(z),\quad z\in  \P\setminus \overline{\Omega}.
\end{cases}
\end{equation} 
The jump over $\partial\Omega$, from $-f_+$ to $f_-$, equals $f_- -(-f_+)=f_++f_-$, which is $f$ itself, assuming that the boundary values exist. The distributional $\partial/\partial\bar{z}$
derivative of ${\rm C}[\mu]$ arising from this jump is of course just $d\mu=\frac{1}{2\I}fdz$, when this is regarded as a measure on $\partial\Omega$.

Quite often the situation will be that each of $f_+$ and $f_-$ have analytic continuations across $\partial\Omega$. Then $f=f_++f_-$ defines $f$ as an analytic function in a full neighborhood
of $\partial\Omega$. Assuming $0\in\Omega$, for simplicity of notation, the germs of $f_+$ and $f_-$ at the origin and infinity, respectively, can always be expanded in power series, which we write as
$$
f_+(z)=  \sum_{k< 0} \frac{a_k}{z^{k+1}}, \quad f_-(z)=\sum_{k\geq 0} \frac{a_k}{z^{k+1}}.
$$ 
This gives a  formal Laurent series representing $f$:
$$
f(z)= \sum_{k\in\Z} \frac{a_k}{z^{k+1}}.
$$
This series may, in good cases, converge in a neighborhood of $\partial\Omega$, and in 
exceptional cases even in all $\C\setminus\{0\}$, but we shall basically consider it as just a formal series. 
The series for $f_+$ and $f_-$ always converge in some neighborhoods of the origin and infinity, respectively, 
but note that the series at infinity does not determine all of $f_-$ in case $\Omega$ is multiply connected. 

An important  example of the above is obtained by taking $f(z)=\bar{z}$, i.e. $d\mu(z)=\frac{1}{2\I}\bar{z}dz$ on $\partial\Omega$. If this boundary is
analytic, then the function $f(z)$, when analytically extended, is the Schwarz function $S(z)$ for $\partial\Omega$, and the coefficients
are the interior and exterior harmonic moments of $\Omega$. We then write the series as
$$
S(z)=\sum_{k\in\Z}\frac{M_k}{z^{k+1}},
$$
with 
$$
S_+(z)=\sum_{k<0}\frac{M_k}{z^{k+1}}, \quad S_-(z)=\sum_{k\geq 0}\frac{M_k}{z^{k+1}},
$$
the $M_k$ being the harmonic moments, interior and exterior, of $\Omega$.

When $z\in \P\setminus{\overline{\Omega}}$, and with the same $\mu$ as above,
$$
{\rm C}[\mu](z)=-\frac{1}{2\pi \I}\int_{\partial\Omega}\frac{\bar{\zeta}d\zeta}{\zeta-z}=-\frac{1}{2\pi \I}\int_{\Omega}\frac{d\bar{\zeta}d\zeta}{\zeta-z}=
$$
\begin{equation}\label{CSM}
={\rm C}[{{\bf 1}}](z)=S_-(z)=\sum_{k\geq 0}\frac{M_k}{z^{k+1}},
\end{equation}
where the series gives the germ at infinity.  As a { matter of notation}, we shall use $S(z)$ to denote the function
$S(z)=\bar{z}$ on $\partial\Omega$ even when that boundary is not analytic. This means that on using the notation (\ref{plusminus}) we always have
$S_\pm(z)$ (or $\bar{z}_{\pm}$) defined as interior and exterior analytic functions, explicitly
\begin{equation}\label{Splusminus}
{\rm C}[\frac{1}{2\I}\bar{\zeta}d\zeta] (z)=-\frac{1}{2\pi \I}\int_{\partial\Omega}\frac{S(\zeta)d\zeta}{\zeta-z}=
\begin{cases}
\quad -S_+(z), \quad z\in {\Omega},\\ \quad\,\,\, S_-(z),\quad z\in  \P\setminus \overline{\Omega}.
\end{cases}
\end{equation} 
By (\ref{CSM}), $S_-={\rm C}^{\rm ext}[{\bf 1}]={\rm C}[{\bf 1}]\big|_{\P\setminus\overline{\Omega}}$, and we will use $S_-$ simply as a brief notation for this Cauchy transform, 
outside $\Omega$, of the characteristic function of $\Omega$.
Similarly, $S_+$ can be identified with a renormalized version (in $\Omega$) of the Cauchy transform of the characteristic function of $\P\setminus \Omega$. 





\subsection{Operators}\label{sec:operators} 

The cohyponormal shift operator ${\rm Z}: \mathcal{H}(\Omega)\to \mathcal{H}(\Omega)$, multiplication by the independent variable $z$, corresponds in the Cauchy transform picture to the operator 
\begin{equation}\label{composition}
{\rm Z}_\mathcal{O}={\rm C}^{\rm ext}\circ{\rm Z}\circ [\frac{\partial}{\partial\bar{z}}]:\quad\overline{\mathcal{O}}(\P\setminus \Omega)_0\to \overline{\mathcal{O}}(\P\setminus \Omega)_0
\end{equation} 
given in terms of an arbitrary continuation $\tilde{f}$ of $f\in \overline{\mathcal{O}}(\P\setminus \Omega)_0$ to all of $\P$ by 
\begin{equation}\label{Z}
{\rm Z}_\mathcal{O}[f]= {\rm C}^{\rm ext}[z\frac{\partial\tilde{f}}{\partial\bar{z}}] ={\rm C}^{\rm ext}[\frac{\partial}{\partial{\bar{z}}}({z}\tilde{f}(z))].
\end{equation}
Since $z\tilde{f}(z)$ is already analytic outside $\Omega$, the effect of $\partial/\partial\bar{z}$ followed by ${\rm C}^{\rm ext}$ is just to adapt by an additive constant to the
normalization at infinity. Thus
\begin{equation}\label{ZCRes}
{\rm Z}_\mathcal{O}[f]={\rm C}^{\rm ext}[\frac{\partial}{\partial{\bar{z}}}({z}\tilde{f}(z))]= zf(z)+ \res_{\infty} f= (zf(z))_-\,\,,
\end{equation}
where the residue at infinity is, by definition, 
$$
\res_{\infty} f=\res_{\zeta=\infty} f(\zeta)d\zeta= -\frac{1}{2\pi\I} \oint_{|\zeta|=R>>1} f(\zeta) d\zeta.
$$

The adjoint  ${\rm Z}^*_\mathcal{O}$ is given by a formula which is a natural counterpart of (\ref{ZCRes}),
differing only by a conjugation. However, the proof of this is not entirely  trivial,  it depends on specific 
continuation properties of the exponential transform.  
One instance of this is that $H(z,w)$ extends holomorphically-\-antiholomorphically
to the entire Riemann sphere, in each variable separately, provided it is viewed as a section of a suitable holomorphic
line bundle (assuming for this statement that $\partial\Omega$ is analytic). This may be compared with somewhat analogues, but different, 
continuation properties for the Bergman and Szeg\"o  kernels:
both of these extend to the Schottky double of the domain, the Bergman kernel as a meromorphic section of the 
canonical bundle, and the Szeg\"o kernel as a meromorphic section of a square root of the canonical bundle.

These extension properties for the exponential transform have been discussed in detail in  \cite{Gustafsson-Putinar-2017},  \cite{Gustafsson-Putinar-2018}, 
but as we will need them in several forthcoming  proofs we compile the necessary identities also here.
For this we introduce a couple of new functions $G(z,w)$, $G^\ast(z,w)$, denote by $F(z,w)$ the restriction of $E(z,w)$ to the exterior, and recall $H(z,w)$. 
All of these functions are  analytic in $z$, anti-analytic in $w$, in their domains of definition, and we have 
\begin{align*}\label{G}
&F(z,w)=E(z,w),\quad &&z\in\C\setminus\ol{\Omega}, \quad w\in\C\setminus\ol{\Omega},\\
&G(z,w)=\frac{\partial E(z,w)}{\partial\bar z}=\frac{E(z,w)}{\bar z-\bar w}, \quad &&z\in\Omega, \,\, w\in\C\setminus\ol{\Omega},\\
&G^*(z,w)=\frac{\partial E(z,w)}{\partial w}=-\frac{E(z,w)}{ z- w},\quad && z\in\C\setminus\ol{\Omega}, \,\, w\in\Omega,\\
&H(z,w)=-\frac{\partial^2 E(z,w)}{\partial\bar z\partial w}=\frac{E(z,w)}{(z-w)(\bar{z}-\bar{w})},\quad && z\in\Omega, \,\, w\in\Omega.
\end{align*}
Here the first equality one each line can be taken as the definition of the quantity to the left.
Note that $G^*(z,w)=\overline{G(w,z)}$, and also that we have the asymptotics
\begin{equation}\label{Gasymptotics}
G(z,w)= -\frac{1}{\bar{w}} + \mathcal{O}{\frac{1}{|w|^2}}, \quad w\to \infty,
\end{equation}
similarly for $G^\ast(z,w)$.

The main point is that the matching properties implicit in the above relations determine the functions uniquely. 
For example, for fixed $z\in\Omega$ the two anti-analytic functions 
in $w$, $H(z,w)$ and $G(z,w)$, match by the Riemann-Hilbert type relation
\begin{equation}\label{HGRH}
H(z,w)(z-w)= G(z,w), \quad z\in \Omega, \,w\in \partial\Omega,
\end{equation}
which together with (\ref{Gasymptotics}) determines $H(z,w)$ and $G(z,w)$ uniquely. Similar statements apply to the relations
\begin{equation}\label{HzaG}
H(z,a)(\bar{z}-\bar{w})=-{G^ \ast(z,w)}, \quad z\in\partial\Omega, \,\, w\in {\Omega},
\end{equation}
\begin{equation}\label{EzwG}
G(z,w)(\bar{z}-\bar{w})={F(z,w)}, \quad\, z\in\partial\Omega, \,\, w\in \P\setminus \ol{\Omega}.
\end{equation}

If $\partial\Omega$ is analytic then all functions $F$, $G$, $G^\ast$, $H$ have analytic continuations across $\partial\Omega$, and  
(\ref{HGRH}) (for example) can be rewritten as
\begin{equation}\label{HG}
H(z,w)(z-\overline{S(w)})= G(z,w),  
\end{equation}
which then holds for $w$ in a full neighborhood of $\partial\Omega$. This relation expresses that, as functions of $w$, the pair $H(z,w)$, $G(z,w)$ makes 
up a global anti-holomorphic section of the line bundle defined by transition function $z-\ol{S(w)}$ (with $z\in\Omega$ kept fixed). 
Again, together with (\ref{Gasymptotics}) this determines $H(z,w)$, $G(z,w)$ uniquely,  as is immediate from  general Riemann surface theory.

All forthcoming theorems will be based solely on the above relations, the original definition of the exponential transform will no longer be needed. 

\begin{theorem}\label{thm:adjoint}
The adjoint  of ${\rm Z}_\mathcal{O}$ is given by 
\begin{equation}\label{adjoint}
{\rm Z}^*_\mathcal{O}[f]= {\rm C}^{\rm ext}[\frac{\partial}{\partial\bar{z}}(\bar{z}\tilde{f}(z))]=(\bar{z}f(z))_- 
\end{equation}
\end{theorem}
In the middle member of (\ref{adjoint}) it is understood that the Cauchy transform is taken only of the restriction of $\frac{\partial}{\partial\bar{z}}(\bar{z}\tilde{f}(z))$ to
$\overline{\Omega}$. In the corresponding formula (\ref{Z}) this need not be said explicitly since $\frac{\partial}{\partial{\bar{z}}}({z}\tilde{f}(z))$ vanishes automatically
outside $\Omega$.

\begin{proof} 
We first confirm the second equality in (\ref{adjoint}) by the following computation, valid for $z\in\P\setminus\overline{\Omega}$:
$$
 {\rm C}^{\rm ext}[\frac{\partial}{\partial\bar{z}}(\bar{z}\tilde{f}(z))](z)= -\frac{1}{2\pi \I}\int_{\Omega}\frac{\frac{\partial}{{\partial\bar{\zeta}}}(\bar{\zeta} \tilde{f}(\zeta))d\bar{\zeta}d\zeta }{\zeta-z}=
$$
$$
= -\frac{1}{2\pi \I}\int_{\partial\Omega}\frac{\bar{\zeta} \tilde{f}(\zeta)d\zeta }{\zeta-z}=(\bar{z}f(z))_-
$$

It remains to prove that ${\rm Z}^*_\mathcal{O}[f]=(\bar{z}f(z))_- $. This follows on using  (\ref{HGRH}), (\ref{Gasymptotics})
together with the fact that any function of the form $\varphi_+$  is analytic in $\Omega$:
$$
\langle (z \tilde{f}(z))_-, \tilde{g}(z)\rangle_{\mathcal{O}}-\langle  \tilde{f}(z), (\bar{z}\tilde{g}(z))_-\rangle_{\mathcal{O}}=
$$
$$
=\frac{1}{4\pi^2}\int_{\partial\Omega}\int_{\partial\Omega} H(z,w)((z \tilde{f} (z))_-)\cdot \overline{\tilde{g}(w)} \,dzd\bar{w}-
$$
$$
-\frac{1}{4\pi^2}\int_{\partial\Omega}\int_{\partial\Omega} H(z,w) \tilde{f} (z)\cdot ((\overline{\bar{w}\tilde{g}(w)})_-) \,dzd\bar{w}=
$$
$$
=\frac{1}{4\pi^2}\int_{\partial\Omega}\int_{\partial\Omega} H(z,w)\left((z \tilde{f} (z))_+ +(z\tilde{f}(z))_-\right) \overline{\tilde{g}(w)} \,dzd\bar{w}-
$$
$$
-\frac{1}{4\pi^2}\int_{\partial\Omega}\int_{\partial\Omega} H(z,w) \tilde{f} (z)\, \left(\overline{\bar{w}\tilde{g}(w)})_+ +(\overline{\bar{w}\tilde{g}(w)})_- \right)dzd\bar{w}=
$$
$$
=\frac{1}{4\pi^2}\int_{\partial\Omega}\int_{\partial\Omega} H(z,w)(z-w) \tilde{f} (z)\, \overline{\tilde{g}(w)} \,dzd\bar{w}=
$$
$$
=\frac{1}{4\pi^2}\int_{\partial\Omega}\int_{\partial\Omega} G(z,w) \tilde{f} (z)\, \overline{\tilde{g}(w)} \,dzd\bar{w}=0.
$$
In the last step we used that $G(z,w)\overline{\tilde{g}(w)}$ is anti-analytic in $\P\setminus \Omega$ and vanishes as $|w|^{-2}$ when $|w|\to\infty$.

\end{proof}

The formula (\ref{adjoint}) can be understood in terms of the decomposition 
\begin{equation}\label{sum}
\frac{\partial}{\partial\bar{z}}(\bar{z}\tilde{f}(z))=\tilde{z}\frac{\partial\tilde{f}(\bar{z})}{\partial \bar{z}}+\tilde{f}(z),
\end{equation}
which represents a common way of writing the adjoint of ${\rm Z}$ (in the $\mathcal{H}(\Omega)$ picture) as a sum of a shift  by $\bar{z}$ and a Cauchy transform
(see \cite{Gustafsson-Putinar-2017}). However, it turns out that this decomposition cannot be pushed
down to the quotient space $\mathcal{H}(\Omega)$ because the balance between the terms depends on which elements are chosen in the equivalence classes.

As an example, let $\varphi$ be a smooth function with compact support in $\Omega$. Then $[\frac{\partial\varphi}{\partial\bar{z}}]=0$ as an element in $\mathcal{H}(\Omega)$,
while still $[\varphi]\ne 0$ for most choices of $\varphi$. The formula for ${\rm Z}^*_\mathcal{O}[f]$ involves $\tilde{f}$ in its role of representing $[\frac{\partial\tilde{f}}{\partial\bar{z}}]\in\mathcal{H}(\Omega)$. 
Now we have $[\frac{\partial\tilde{f}}{\partial\bar{z}}]=[\frac{\partial\tilde{f}}{\partial\bar{z}}]+[\frac{\partial\varphi}{\partial\bar{z}}]$, and if choose $\tilde{f}+\varphi$ in place of $\tilde{f}$
in (\ref{sum}) we get 
\begin{equation}\label{sum1}
\frac{\partial}{\partial\bar{z}}(\bar{z}\tilde{f}(z))+\frac{\partial}{\partial\bar{z}}(\bar{z}\varphi(z))=\tilde{z}\frac{\partial(\tilde{f}(\bar{z})+\varphi(z))}{\partial \bar{z}}+(\tilde{f}(z)+\varphi(z)).
\end{equation}
Here $\bar{z}\varphi(z)$ has compact support in $\Omega$, so the second term in the left member represents the zero element in $\mathcal{H}(\Omega)$, but in the right member there 
has been a shift of balance between the two terms, even as elements in $\mathcal{H}(\Omega)$.

The original operator ${\rm Z}$ on $\mathcal{H}(\Omega)$ was a shift operator, and also its counterpart ${\rm Z}_\mathcal{O}$ is a shift, now on the Laurent series: with
$$
f(z)=\sum_{k=0}^\infty\frac{a_k}{z^{k+1}}, \quad |z|>>1,
$$
we have 
\begin{equation}\label{Zstar}
{\rm Z}_\mathcal{O}[f](z)= (zf(z))_- =\sum_{k=0}^\infty\frac{a_{k+1}}{z^{k+1}},\\
\end{equation}
However, ${\rm Z}_\mathcal{O}^*$ is not the corresponding backward shift, it has a more complicated description.  See for example (\ref{SS}) below.

\begin{example}\label{ex:1}
The identity function ${\bf 1}=[1]\in\mathcal{H}(\Omega)$ plays an important role in the theory, starting with the observation that the commutator between ${\rm Z}$ and ${\rm Z}^*$ is a projection onto the
subspace spanned by ${\bf 1}$:
\begin{equation}\label{ZZ}
[{\rm Z}, {\rm Z}^*] ={\bf 1}\otimes {\bf 1}.
\end{equation}

In $ \overline{\mathcal{O}}(\P\setminus \Omega)_0$ this function ${\bf 1}$ is represented by its Cauchy transform, namely
$S_-={\rm C}^{\rm ext}[{\bf 1}]$, see (\ref{CSM}). The most natural extension of it to all of $\P$ is the full Cauchy transform
$$
{\rm C}[{\bf 1}](z) = -\frac{1}{\pi} \int_\Omega \frac{dA(\zeta)}{\zeta-z},
$$
which as in Section~\ref{sec:intext} spells out as 
\begin{equation}\label{CSzS}
{\rm C}[{\bf 1}](z) =
\begin{cases}
\bar{z}-S_+(z),\quad &z\in\Omega,\\
S_-(z),\quad &z\in\P\setminus \Omega.
\end{cases}
\end{equation}
Thus, applying (\ref{ZCRes}), ${\rm Z}{\bf 1}=[z]$ corresponds in $\overline{\mathcal{O}}(\P\setminus \Omega)_0$ to $({\rm Z}_\mathcal{O} S_-)(z)=(zS_-(z))_-$ .

The operator ${\bf 1}\otimes {\bf 1}$ on $\mathcal{H}(\Omega)$ becomes, as an operator on $\overline{\mathcal{O}}(\P\setminus \Omega)_0$,
$$
({S_-}\otimes {S_-})_\mathcal{O}: \overline{\mathcal{O}}(\P\setminus \Omega)_0\to \overline{\mathcal{O}}(\P\setminus \Omega)_0,
$$
given by
\begin{equation}\label{ZZstar}
({S_-}\otimes {S_-})_\mathcal{O}[f]=\langle f,S_-\rangle_\mathcal{O} \cdot S_-\,\,,
\end{equation}
and one may in a certain sense identify (or represent) it with the kernel $S_-(\zeta)S_-(z)$.
See Section~\ref{sec:commutator} for the proof.
\end{example}

\begin{example}\label{ex:zn}
Repeating the above example more generally for $[z^n]={\rm Z}^n {\bf 1}\in\mathcal{H}(\Omega)$, $n\geq 0$, gives
that it in $\overline{\mathcal{O}}(\P\setminus \Omega)_0$ is represented by 
$$
{\rm Z}_\mathcal{O}^n[S_-](z)=(z^nS_-(z))_-=(z^n S(z))_-
$$
with power series expansion at infinity given by
$$
(z^nS(z))_-= \frac{M_n}{z}+ \frac{M_{n+1}}{z^2}+  \frac{M_{n+2}}{z^3}+\cdots.
$$
Thus ${\rm Z}^n$ acts here as an $n$ step shift operator of the harmonic moments.
\end{example}

\begin{example}\label{ex:Zstar1}
Next we consider ${\rm Z}^*{\bf 1}\in\mathcal{H}(\Omega)$. Its counterpart in $\overline{\mathcal{O}}(\P\setminus \Omega)_0$
is ${\rm Z}_\mathcal{O}^*[S_-](z)$, and by Theorem~\ref{thm:adjoint} this is given by
\begin{equation}\label{ZFSS}
{\rm Z}_\mathcal{O}^*[S_-](z)=(\bar{z}S_-(z))_-=(S(z)S_-(z))_-\,\,.
\end{equation}
In terms of power series, ${\rm Z}_\mathcal{O}^*[S_-]$ is represented by the strictly negative powers in the product of the power series of $S(z)$
and of $S_-(z)$:
\begin{equation*}\label{SS0}
{\rm Z}_\mathcal{O}^*[S_-](z)=(S(z)S_-(z))_-    
=
\end{equation*}
\begin{equation}\label{SS}
=\left(\sum_{j\in \Z}\frac{M_j}{z^{{j+1}}}\cdot\sum_{k\geq 0}\frac{M_k}{z^{k+1}}\right)_- =\sum_{k\geq 0}(\sum_{j\geq 0}M_{k-j-1}M_j)\frac{1}{z^{k+1}}, 
\end{equation}
to compare with
$$
{\rm Z}_\mathcal{O}[S_-](z)=(zS_-(z))_-=\sum_{k\geq 0}\frac{M_{k+1}}{z^{k+1}}
$$
from Example~\ref{ex:zn}
\end{example}

\begin{example}\label{ex:cauchykernel}
The Cauchy kernel in $\mathcal{H}(\Omega)$,
$$
[\frac{1}{z-a}]=({\rm Z}-a)^{-1} {\bf 1},
$$
corresponds in the Cauchy transform picture to 
$$
({\rm Z_\mathcal{O}}-a)^{-1} [{S_-}](z) 
=-\frac{1}{2\pi \I} \int_\Omega \frac{d\bar{\zeta}d\zeta}{(\zeta-z)(\zeta-a)}
$$
$$
=-\frac{1}{2\pi \I(z-a)} \left(\int_{\Omega} \frac{d\bar{\zeta}d\zeta}{\zeta-z}- \int_{\Omega}\frac{d\bar{\zeta}d\zeta}{\zeta-a}\right)
=\frac{{\rm C}[{\bf 1}](z)-{\rm C}[{\bf 1}](a)}{z-a}.
$$
Here $z\in\C\setminus\ol{\Omega}$, but any $a\in\C$ is allowed. If $a\in\P\setminus\ol{\Omega}$ the above becomes
$({S_-(z)-S_-(a)})/({z-a})$, and more generally we have, for $f\in\ol{\mathcal{O}}(\P\setminus\Omega)_0$ and $z,a\in \P\setminus\ol{\Omega}$,
\begin{equation}\label{Za}
({\rm Z_\mathcal{O}}-a)^{-1} [{f}](z)=\frac{f(z)-f(a)}{z-a}.
\end{equation}
\end{example}

\begin{example}\label{ex.GH}
As a complement to the relation ${\rm C}^{\rm ext}[{\bf 1}] =S_-$ \,\,in Example~\ref{ex:1} we mention the following relations,
holding for $z\in\P\setminus \ol{\Omega}$:
\begin{align*}
{\rm C}^{\rm ext}[{G(\cdot, w)}](z)&=E(z,w)-1, \quad &&\text{for fixed }w\in\P\setminus \ol{\Omega},\\
{\rm C}^{\rm ext}[{H(\cdot, w)}](z)&=-G^{\ast}(z,w), \quad &&\text{for fixed }w\in{\Omega}.
\end{align*}
See Section~2.5 in \cite{Gustafsson-Putinar-2017} for the simple proofs.
\end{example}



\subsection{Boundary operations}\label{sec:boundary}

We return to the boundary version of the  inner product (\ref{innerproduct1}):
\begin{equation}\label{innerproduct2}
\langle f,g\rangle_\mathcal{O}= \frac{1}{4\pi^2} \int_{\partial{\Omega}} \int_{\partial{\Omega}} H(z,w)f(z)\overline{g(w)}dzd\bar{w}.
\end{equation}
%
%
%
We have so far used $\tilde{f}$ to denote various ``soft'' extensions of a given  $f\in{O}(\P\setminus\Omega)_0$ to all of $\P$.
Now working on the boundary suggests a more drastic way: simply extend it by zero in $\Omega$:
\begin{equation}\label{fzero}
\tilde{f}(z)=
\begin{cases}
f(z), \quad &z\in \P\setminus\Omega,\\
0, \quad&z\in\Omega.
\end{cases}
\end{equation}
Then $\frac{\partial\tilde{f}}{\partial\bar{z}}=\mu$, where $\mu$ is the distribution supported by $\partial\Omega$ and given there by $d\mu(z)= \frac{1}{2\I}f(z)dz$,
as in Section~\ref{sec:intext}.

With $\tilde{f}$ chosen to vanish in $\Omega$ we have, choosing  (\ref{Z}) as an example, that $\frac{\partial}{\partial{\bar{z}}}({z}\tilde{f}(z))$ becomes the measure
$\frac{1}{2\I}zf(z)dz$ on $\partial\Omega$. Similarly for other cases. This
gives the following simple and well balanced expressions, where $z\in\P\setminus\overline{\Omega}$:
\begin{align}\label{cauchyshift}
{\rm Z}_\mathcal{O}[f](z)&= -\frac{1}{2\pi\I}\int_{\partial\Omega}\frac{\zeta f({\zeta})d\zeta}{\zeta-z},\\
{\rm Z}^*_\mathcal{O}[f](z)&= -\frac{1}{2\pi\I}\int_{\partial\Omega}\frac{\bar{\zeta} f({\zeta})d\zeta}{\zeta-z}.
\end{align}

The relationship between the adjoint pair ${\rm Z}_\mathcal{O}$ and ${\rm Z}^*_\mathcal{O}$ here appears very natural, with Cauchy integrals containing
multiplications with $\zeta$ and $\bar\zeta$, respectively. 
One might believe that this is something general, but it depends on specific
properties of the kernel $H(z,w)$. As can be seen from the final steps in the proof of
Theorem~\ref{thm:adjoint}, the kernel $H(z,w)$ is actually characterized (up to a constant factor) by the two operators defined by the right members
of (\ref{cauchyshift}) being adjoint to each other. 

We may now collect (from (\ref{Zstar}), (\ref{ZFSS}), (\ref{Za})), and somewhat extend, the functional calculus for ${\rm Z}_\mathcal{O}$ as follows:
\begin{align*}\label{Zaf}
{\rm Z}_\mathcal{O}[f](z)&=(zf(z))_- \,\,,\\
{\rm Z}^*_\mathcal{O}[f](z)&=(\bar{z}f(z))_-=(S(z)f(z))_-\,\,,\\
({\rm Z_\mathcal{O}}-a)^{-1} [f](z)&= -\frac{1}{2\pi\I}\int_{\partial\Omega}\frac{ f({\zeta})d\zeta}{(\zeta-z)(\zeta-a)}=\frac{f(z)-f(a)}{z-a},\\
({\rm Z^\ast_\mathcal{O}}-\bar{a})^{-1} [f](z)&= -\frac{1}{2\pi\I}\int_{\partial\Omega}\frac{ f({\zeta})d\zeta}{(\zeta-z)(\bar{\zeta}-\bar{a})}.
\end{align*}
Here $f\in\ol{\mathcal{O}}(\P\setminus\Omega)_0$ and $z,a\in \P\setminus\ol{\Omega}$.


\subsection{The commutator}\label{sec:commutator}

The aim of the present section is to confirm formula (\ref{ZZstar}) for the commutator.  To this aim we prove a few preparatory results.

\begin{lemma}\label{lem:intHS}
For any fixed $a\in\Omega$ and $w\in \P\setminus\ol{\Omega}$ there holds
\begin{equation}\label{intHS} 
\frac{1}{2\pi \I}\int_{\partial\Omega} H(z,a)S_-(z)dz =1,
\end{equation}
\begin{equation}\label{intHE} 
\frac{1}{2\pi \I}\int_{\partial\Omega}\frac{H(z,a)dz}{E(z,w)}=\frac{1}{\bar{w}-\bar{a}}.
\end{equation}
\end{lemma}

\begin{remark}\label{rem:intHS}
Equation (\ref{intHS}) equivalently says that $\frac{1}{\pi}\int_\Omega H(z,a)dA(z)=1$.
\end{remark}

\begin{proof}
Since $S_+(z)-\bar{a}$ is holomorphic in $\Omega$ we have, using (\ref{HzaG}),
$$
\int_{\partial\Omega} H(z,a)S_-(z)dz=\int_{\partial\Omega} H(z,a)(S_-(z)+S_+(z)-\bar{a})dz=
$$
$$
=\int_{\partial\Omega} H(z,a)(\bar{z}-\bar{a})dz=-\int_{\partial\Omega} G^\ast(z,a)dz
=\res_\infty G^\ast (\cdot,a)=2\pi \I,
$$
proving (\ref{intHS}). As for (\ref{intHE}), we multiply with $\bar{w}-\bar{a}$ and compute
$$
(\bar{w}-\bar{a})\int_{\partial\Omega}\frac{H(z,a)dz}{E(z,w)}=\int_{\partial\Omega}\frac{(\bar{w}-\bar{z}+\bar{z}-\bar{a})H(z,a)dz}{E(z,w)}=
$$
$$
=-\int_{\partial\Omega}\frac{(\bar{z}-\bar{w})H(z,a)dz}{F(z,w)}+\int_{\partial\Omega}\frac{(\bar{z}-\bar{a})H(z,a)dz}{F(z,w)}=
$$
$$
=-\int_{\partial\Omega}\frac{H(z,a)dz}{G(z,w)}+\int_{\partial\Omega}\frac{G^\ast(z,a)dz}{F(z,w)}=0+ \int_{|z|=R>>1}\frac{dz}{z}=2\pi \I.
$$
In the last steps,  (\ref{EzwG}), (\ref{HzaG}) and the asymptotic  properties of $G$ and $G^\ast$ in $\Omega$ and $\P\setminus\Omega$, respectively, were used
(see (\ref{Gasymptotics})). The identity obtained is  exactly the desired one, namely (\ref{intHE}).
\end{proof}

\begin{lemma}\label{lem:fS}
For $f\in\mathcal{O}(\P\setminus\Omega)_0$ we have
\begin{equation}\label{fS}
\langle f,S_-\rangle_\mathcal{O} = \res_\infty f . 
\end{equation}
\end{lemma}

\begin{remark}\label{rem:f1}
The counterpart of (\ref{fS}) in $\mathcal{H}(\Omega)$ is the formula
$$
\langle f,{\bf 1}\rangle_{\mathcal{H}(\Omega)}= \frac{1}{\pi} \int_\Omega f dA,
$$
for $f\in\mathcal{H}(\Omega)$, which follows immediately from the identity in Remark~\ref{rem:intHS}.
\end{remark}

\begin{proof}
Once again we invoke the continuation properties  of $H(z,w)$, this time (\ref{Gasymptotics}), (\ref{HG}). Using also the fact that $\ol{S(w)}_+ -z$ is 
anti-analytic in $\Omega$, as a function of $w$, we obtain
$$
\langle f,S_-\rangle_\mathcal{O}=\frac{1}{4\pi^2}\int_{\partial\Omega}\int_{\partial\Omega}H(z,w)f(z)\overline{S_-(w)} dzd\bar{w}=
$$
$$
=\frac{1}{4\pi^2}\int_{\partial\Omega}\left(\int_{\partial\Omega}H(z,w)\left(\overline{S_+(w)}+\ol{S_-(w)}-z\right) d\bar{w}\right)f(z)dz=
$$
$$
=\frac{1}{4\pi^2}\int_{\partial\Omega}\left( \int_{\partial\Omega}H(z,w)(w-z) d\bar{w}\right)f(z)dz=
$$
$$
=\frac{1}{2\pi\I}\int_{\partial\Omega}(\frac{1}{2\pi\I}\int_{\partial\Omega}G(z,w)d\bar{w})f(z)dz=-\frac{1}{2\pi\I}\int_{\partial\Omega}f(z)dz.
$$
\end{proof}

\begin{lemma}\label{lem:commutator}
For $f\in\mathcal{O}(\P\setminus\Omega)_0$ we have
\begin{equation}\label{commutator}
[{\rm Z}_\mathcal{O}, {\rm Z}_\mathcal{O}^*]f= (\res_\infty f) \cdot S_-.
\end{equation}
\end{lemma}

\begin{proof}
The proof is a straight-forward computation using (\ref{ZCRes}) and (\ref{adjoint}):
$$
{\rm Z}_\mathcal{O}^*{\rm Z}_\mathcal{O}f=(S(z)(zf(z))_-)_-=(S(z)zf(z))_-+(\res_\infty f) \cdot S_-(z),
$$
$$
{\rm Z}_\mathcal{O}{\rm Z}_\mathcal{O}^*f=(z(S(z)f(z))_-)_-=z(S(z)f(z))_- +\res_\infty (Sf)_-
$$
$$
= (zS(z)f(z))_- -\res_\infty (Sf)_- \,  +\res_\infty (Sf)_-  =(zS(z)f(z))_- \,\,.
$$
Thus
$$
[{\rm Z}_\mathcal{O}, {\rm Z}_\mathcal{O}^*]f= -\frac{1}{2\pi \I} \int_{\partial\Omega} f(z)dz \cdot S_- \,.
$$
This is the same as (\ref{commutator}).
\end{proof}

Combining the two lemmas gives $[{\rm Z}_\mathcal{O}, {\rm Z}_\mathcal{O}^*]f= \langle f,S_-\rangle_\mathcal{O}\cdot S_-\,$.
Thus 

\begin{theorem}
$$
[{\rm Z}_\mathcal{O},{\rm Z}_\mathcal{O}^*]=S_-\otimes S_-
$$
\end{theorem}


\subsection{The reproducing kernel}\label{sec:kernel}

As it has been indicated before,  point evaluations are continuous linear functionals on $\overline{\mathcal{O}}(\P\setminus \Omega)_0$, hence that space has a
reproducing kernel, which we now make explicit.

\begin{theorem}\label{thm:reproducing}
For any $w\in\P\setminus\ol{\Omega}$ the function 
\begin{equation}\label{reproducing}
L_w(z)=L(z,w)= \frac{1}{E(z,w)}-1 \quad (z\in\P\setminus\ol{\Omega})
\end{equation}
is in $\overline{\mathcal{O}}(\P\setminus \Omega)_0$ and has the reproducing property
$$
f(w)=\langle f, L_w\rangle_{\mathcal{O}}, \quad f\in \overline{\mathcal{O}}(\P\setminus \Omega)_0.
$$
\end{theorem}

\begin{proof}
It is enough to consider test functions of the form 
\begin{equation}\label{fza}
f(z)=k_a(z)=\frac{1}{z-a}
\end{equation}
with $a\in \Omega$, since linear combinations of these are dense in $\overline{\mathcal{O}}(\P\setminus \Omega)_0$.
We remark in passing that the correlation kernel for such functions is exactly $H$:
$$
\langle \frac{1}{z-a}, \frac{1}{z-b}\rangle_\mathcal{O} = H(a,b), \ \ a,b \in \Omega.
$$
This corresponds to the inner product in $\mathcal{H}(\Omega)$ of Dirac masses at $a$ and $b$.

Now, with $f$ of the form (\ref{fza}) we need to show that
\begin{equation}\label{repr}
\langle \frac{1}{z-a}, \frac{1}{E(z.w)}-1\rangle_\mathcal{O}=\frac{1}{w-a}.
\end{equation}
The proof of this is an easy consequence of (\ref{intHE}) in Lemma~\ref{lem:intHS}:
$$
\langle \frac{1}{z-a}, \frac{1}{E(z.w)}-1\rangle_\mathcal{O}=
$$
$$
=\frac{1}{4\pi^2}  \int_{\partial\Omega}\int_{\partial\Omega}H(u,v)\frac{du}{u-a}\, \ol{(\frac{1}{E(v,w)}-1)}\,d\bar{v}=
$$
$$
 ={\frac{2\pi \I}{4\pi^2} \int_{\partial\Omega}H(a,v)\frac{d\bar{v}}{\ol{E(v,w)}})}=\ol{\frac{1}{2\pi\I}\int_{\partial\Omega}\frac{H(v,a)dv}{E(v,w)}}=\frac{1}{{w}-{a}}.
$$
\end{proof}


\section{On the nature of the functional model}\label{sec:nature}

\subsection{Characterization of null elements in $\mathcal{H}(\Omega)$}

The above discussions give a possibility to better understand the space $\mathcal{H}(\Omega)$.
Initially this space was built from test functions or distributions in $\Omega$, but 
the equivalence classes when passing to the appropriate quotient space are very big, so it is difficult to build
a function theoretic intuition in this setting. 

To illustrate how far $\mathcal{H}(\Omega)$ is from a function space, consider for example the case $\Omega$ is the unit disk.
In this case the function $f(z)=z$ is the zero element in $\mathcal{H}(\Omega)$, i.e. $\|z\|=0$. 
Even worse, if one multiplies this zero function with $\bar{z}$ it becomes $\bar{z}\cdot z$, which is not zero
in $\mathcal{H}(\Omega)$: $[z]=0$ but $[\bar{z}\cdot z]\ne 0$.
Thus multiplication by $\bar{z}$ is not a continuous operator in $\mathcal{H}(\Omega)$. In fact, it even does not make sense in that quotient space.
As another example:  if $\Omega$ is
an ellipse in standard position then, as elements in $\mathcal{H}(\Omega)$,  $\bar{z}=az$ for some constant $a$ (see Example~\ref{ex:ellipse}  below).

Basically what can now be understood is that it is natural to think of the elements of $\mathcal{H}(\Omega)$ as generators (sources, sinks, etc.)
of  (physical) fields. These generators are then located in $\Omega$ (or its closure), while the fields themselves live in the complement
$\P\setminus \Omega$, and they are more exactly defined as the Cauchy transforms of the sources. A huge amount of different generators
can in this way give rise to the same external field, and this explains the big equivalence classes when forming $\mathcal{H}(\Omega)$.
See Section~\ref{sec:fluid} for some details on a fluid dynamic interpretation. 
When these sources (etc.) are moved to the boundary they however become unique representatives for their elements in $\mathcal{H}(\Omega)$,
as indicated in Section~\ref{sec:boundary}.

The following theorem is essentially implicit in what has already been said, but we state it separately for clarity. 
We do not know of any characterization of what the most general object in $\mathcal{H}(\Omega)$ looks like, so we focus only
on elements which can be represented by distributions $\mu$ in  $\mathcal{H}(\Omega)$
(strictly speaking $[\mu]\in\mathcal{H}(\Omega)$), for which the norm $\|\mu\|$ makes immediate sense. More precisely, we assume that  $\mu$ either is a distribution with
compact support in $\Omega$, or that $d\mu= gdA$ for some function $g\in L^\infty (\Omega)$.

\begin{theorem}\label{thm:zeronorm}
With $\mu\in\mathcal{H}(\Omega)$ as above, the following statements are equivalent.
\begin{itemize}
\item[$(i)$]\quad
$\|\mu\|_{\mathcal{H}(\Omega)}=0$. 
\item[$(ii)$]\quad
${\rm C}^{\rm ext}[\mu]=0$.
\item[$(iii)$]\quad
$\mu \text{ annihilates all analytic functions in } \Omega$.
\item[$(iv)$] \quad
$\mu={\partial f}/{\partial \bar{z}} \text{ for some function } f  \text{ which vanishes on } \P\setminus{\Omega}$.
\end{itemize}
\end{theorem}

\begin{remark}
$(iii)$ means, in the in case $d\mu =g\,dA$, $g\in L^\infty (\Omega)$, for example, that $\int_\Omega fg\,dA=0$ for all analytic $f$ in $L^1(\Omega)$.
The function (or rather distribution) $f$ in $(iv)$ need not be very regular in $\Omega$ (if $\mu$ is just a distribution), but near $\partial\Omega$ it will at least be a 
continuous function, hence it makes good sense to say that it vanishes on $\P\setminus{\Omega}$.
\end{remark}

On setting $k_a(z)=({z-a})^{-1}$, as in (\ref{fza}),        
the Cauchy transform of $\mu \in\mathcal{H}(\Omega)$ can be written (recall Remark~\ref{rem:f1})
\begin{equation}\label{cauchy0}
{\rm C}^{\rm ext}[\mu](a)
=\langle k_a \mu,\mathbf{1} \rangle_{\mathcal{H}(\Omega)}.
\end{equation}

\begin{proof}

$(i)\Rightarrow (ii)$:
If $[\mu]=0$ in $\mathcal{H}(\Omega)$ then also $[k_a\,\mu]=0$ in $\mathcal{H}(\Omega)$, for $a\in \P\setminus\overline{\Omega}$,  because multiplication
with $k_a$ is a bounded operator on $\mathcal{H}(\Omega)$ when $a\in\P\setminus\overline{\Omega}$. 
Hence
$$
{\rm C}^{\rm ext}[\mu] (z)=\langle k_a\,\mu,\mathbf{1} \rangle=0
$$ 
for $z\in \P\setminus\overline{\Omega}$, and by continuity up to $\partial\Omega$.

$(ii)\Rightarrow (iii)$:
This is immediate since the kernels $k_a$, with $a\in \P\setminus\overline{\Omega}$,  span the relevant class of  analytic functions in $\Omega$. 

$(iii)\Rightarrow (ii)$: Choose the analytic function to be $k_a$, $a\in \P\setminus\overline{\Omega}$.

$(iii)\Rightarrow (iv)$: Choose $f={\rm C}^{\rm ext}[\mu]$.

$(iv)\Rightarrow (i)$:
If $\mu={\partial f}/{\partial \bar{z}}$, $f=0$ on $\partial\Omega$, then for fixed $w\in\Omega$
$$
\mu(H(\cdot, w))= \int_\Omega H(z,w)\,\frac{\partial f}{\partial \bar{z}}\,dA(z)=0,
$$
by partial integration and where the left member denotes the action of $\mu$ on $H(\cdot,w)$ as a distribution. 
Notice that $H(\cdot,w)$ is regular up to $\partial\Omega$ for fixed $w\in\Omega$. 
It next follows by iterated integration that
$$
\|\mu\|^2=(\mu\otimes\mu) (H)=0.
$$

\end{proof}

\begin{corollary}
$$
\|\frac{\partial}{\partial \bar{z}}\frac{1}{H(z,z)}\|_{\mathcal{H}(\Omega)}=0.
$$
\end{corollary}

\begin{proof}
This is clear from $(iv)$ in the theorem since $H(z,z)>0$ in $\Omega$, $H(z,z)=0$ on $\partial\Omega$.
\end{proof}

\begin{example}\label{ex:ellipse}
For the standard ellipse with half axes $a>b>0$ we have (see p. 97 in \cite{Gustafsson-Putinar-2017})
$$
H(z,w)=\frac{C}{4a^2b^2 + (a^2-b^2)(z^2+\bar{w}^2) -2(a^2+b^2)z\bar{w}},
$$
where $C=C(a,b)>0$ is a constant. Then
$$
\frac{\partial}{\partial \bar{z}}\frac{C}{H(z,z)}=2(a^2-b^2)\bar{z}-2(a^2+b^2)z,
$$
hence the relation
$$
(a^2-b^2)\bar{z}=(a^2+b^2)z
$$
holds in $\mathcal{H}(\Omega)$.
\end{example}

\subsection{A fluid dynamic interpretation of the functional space}\label{sec:fluid}

We may think of complex valued functions in $\P$ as representing  fluid velocity fields. Dividing into
real and imaginary parts it is then customary to let  $f=u+\I v$ represent the vector field
$$
{\bf v}= u \frac{\partial}{\partial x} -v \frac{\partial}{\partial y},
$$
or, in a dual picture, the $1$-form $\nu= u dx -v dy$. Then the derivative
$$
2\frac{\partial f}{\partial \bar{z}}= \frac{\partial u}{\partial x}+\frac{\partial (-v)}{\partial y}- \I \,(\frac{\partial (- v)}{\partial x}-\frac{\partial u}{\partial y})
$$
can be identified with the (complex-valued) vector field
$$
 {\rm div}\, {\bf v} - \I \,{\rm curl} \, {\bf v},
$$
or with the $2$-form $d*\nu -\I \,d\nu$.
We see in particular that $f$ is analytic if and only if the flow is divergence and curl free, i.e.  incompressible and without vorticity. 

Now we consider smooth complex-valued functions, denoted $\tilde{f}$ in our previous contexts, and vanishing at infinity. We assume that the
corresponding flow is incompressible and free of vorticity outside the closure of $\Omega$. Thus $f=\tilde{f}|_{\P\setminus\Omega}$
is analytic there, say for simplicity in a full neighborhood of $\P\setminus\Omega$, and it is the exterior Cauchy transform of $\partial \tilde{f}/\partial \bar{z}$.

If $\tilde{f}$ is not identically zero in $\P\setminus\Omega$, then there must be sources/sinks and/or vorticity present somewhere, and clearly this must be in $\Omega$.
But many different configurations of sources/sinks and vorticity in $\Omega$ can generate the same flow field  outside $\Omega$. And this
redundancy exactly corresponds  to the big equivalence classes for the space $\mathcal{H}(\Omega)$, i.e. that space can be thought of as
the space of generators of exterior flow fields.  The essence of Section~\ref{sec:boundary} is on the other hand that sources/sinks and vorticity can be pushed to the 
boundary, and then they become unique representatives for the flow.
\bigskip

{\bf Acknowledgement.} The authors thank the referee for pertinent and constructive observations which led to an improvement of an earlier version of this article.




\bibliography{bibliography_gbjorn.bib}

\newpage
{\it Errata for \cite{Gustafsson-Putinar-2017}}
\bigskip

Since the recent notes \cite{Gustafsson-Putinar-2017} are a natural parent of the present article, we list below a few corrections:

\begin{itemize}

\item p.20, equation (2.35): The variables in the integral in the right member are incorrect. Equation (2.35) should read
$$
G^\ast(z,w)=\frac{1}{\pi} \int_\Omega H(u,w)\frac{dA(u)}{u-z}, \quad z\in \Omega^e, \,\,w\in\Omega.
$$

\item Section~3.3, beginning with equations (3.12), (3.13): The operators $\bar{\rm Z}$ and ${\rm C}$ are introduced and it is said that 
$\bar{\rm Z}+{\rm C}={\rm Z}^*$, see equation (3.14) (in \cite{Gustafsson-Putinar-2017}).
However, $\bar{\rm Z}$ and ${\rm C}$ do not make independent sense as operators in that Hilbert space (only the sum does).

The problem is that $\mathcal{H}(\Omega)$ is a quotient space, and $\bar{\rm Z}$, ${\rm C}$ make sense when acting on representatives
for equivalence classes, but the result depends on which representatives are chosen. 
Since almost all essential statements in \cite{Gustafsson-Putinar-2017} are made in terms of the combination $\bar{\rm Z}+{\rm C}$ the mistake has limited consequences. 

\item p.93: The orthonormal vectors should be
$$
e_{nk}=(k+1)z^n\bar{z}^k.
$$
Thus an ON-basis for $\mathcal{H}(\D)$ is
$$
\{e_{00}, e_{01}, e_{02}, \dots\}=\{(k+1)\bar{z}^k: k=0,1,2,\dots\}.
$$

\item p.98, line 3 from below: ${\rm T}^\ast_n$ should be  ${\rm T}_n$.

\end{itemize}

\end{document}